\documentclass[preprint,12pt]{article}
\usepackage{amssymb}
\usepackage{amsmath}
\usepackage{amsthm}
\usepackage{xcolor}

\newtheorem{theorem}{Theorem}
\newtheorem{lemma}[theorem]{Lemma}

\newtheorem{proposition}[theorem]{Proposition}

\begin{document}
\baselineskip 18pt
\title{The zero blocking numbers of generalized Kneser graphs and generalized Johnson graphs}
\author{Hau-Yi Lin\,$^{1,a}$, Wu-Hsiung Lin\,$^{1,b,*}$, Gerard Jennhwa Chang\,$^{2,c}$ \\
{\footnotesize $^1$\,Department of Applied Mathematics, National Yang Ming Chiao Tung University,} \\
{\footnotesize Hsinchu 30010, Taiwan\hspace{8.8cm}} \\
{\footnotesize $^2$\,Department of Mathematics, National Taiwan University,
                    Taipei 10617, Taiwan\hspace{0.4cm}} \\
{\footnotesize $^a$\,Email address: zaq1bgt5cde3mju7@gmail.com\hspace{5.55cm}} \\
{\footnotesize $^b$\,Email address: wuhsiunglin@nctu.edu.tw\hspace{6.3cm}} \\
{\footnotesize $^c$\,Email address: gjchang@math.ntu.edu.tw\hspace{6.25cm}} \\
{\footnotesize $^*$\,Corresponding author.\hspace{9.15cm}} \\}

\date{July 18, 2025}
\maketitle

\noindent
{\bf Abstract.}
This paper extends the results by Afzali, Ghodrati and Maimani for the zero blocking numbers of Kneser graphs and Johnson graphs to generalized Kneser graphs and generalized Johnson graphs.

\bigskip
\noindent
{\bf Keywords.}
Zero forcing process, zero blocking set, zero blocking number, (generalized) Kneser graph, (generalized) Johnson graph.

\bigskip
\noindent
{\bf 2020 Mathematics Subject Classification.}
05C69, 05C85, 68R10.

\section{Introduction}

In a {\it zero forcing process}, vertices of a graph are colored {\it black} or {\it white} initially.
If there exists a black vertex adjacent to exactly one white vertex,
then the white vertex is {\it forced} to be black.
A {\it zero forcing set} is an initial set of black vertices in a zero forcing process that ultimately expands to include all vertices of the graph; otherwise we call it a {\it failed zero forcing set} and its complement, the initial set of white vertices, a {\it zero blocking set}.
The {\it zero forcing number} $Z(G)$ of a graph $G$ is the minimum size of a zero forcing set.
The {\it failed zero forcing number} $F(G)$ of $G$ is the maximum size of a failed zero forcing set.
The {\it zero blocking number} $B(G)$ of $G$ is the minimum size of a zero blocking set.
It is clear that $F(G)+B(G)=|V(G)|$ for every graph $G$.
Hence, studying $F(G)$ and $B(G)$ are equivalent.

The zero forcing process was introduced in \cite{2008AIM} to study minimum rank problems in linear algebra, and independently in \cite{2007bg} to explore quantum systems memory transfers in quantum physics.
The computation of the zero forcing number is NP-hard \cite{2008a}.
Considerable effort has been made to find exact values and bounds for this
number for specific classes of graphs, and to investigate a variety of related concepts arising in zero forcing processes.

The concept of failed zero forcing number was first introduced in \cite{2015fjs}.
The computation of the failed zero forcing number is NP-hard \cite{2017s}.
Results for exact values and bounds for this number were also established in
{\cite{2024agm,2020apnaA,2020apnaB,2016ajps,2020bccknt,2024c,2021grtn,
2024grtn,2020ksv,2023ktn,2023su}.
Our main concern is the results for the zero blocking numbers of Kneser graphs $K(n,k)$ and Johnson graphs $J(n,k)$ given by Afzali, Ghodrati and Maimani \cite{2024agm}.
In particular, we extend their results to generalized Kneser graphs and generalized Johnson graphs.

All graphs in this paper are finite, undirected, without loops and multiple edges.
In a graph, the {\it neighborhood} of a vertex $x$ is the set $N(x)$ of all vertices adjacent to $x$ and the {\it closed neighborhood} of $x$ is $N[x]=\{x\}\cup N(x)$.
A vertex $x$ is {\it isolated} if $N(x)=\emptyset$.
Two distinct vertices $x$ and $y$ are {\it twins} if either $N(x)=N(y)$ or $N[x]=N[y]$.

The following are some basic properties for zero blocking numbers.

\begin{proposition} \label{components}
If a graph $G$ has exactly $r$ components $G_1, G_2, \ldots, G_r$, then $B(G)=\min_{1\le i\le r} B(G_i)$.
\end{proposition}

\begin{proposition}  \label{isolated}
For any graph $G$, we have $B(G)\ge 1$, and $B(G)=1$ if and only if $G$ has an isolated vertex.
\end{proposition}

\begin{proposition} {\bf (\cite{2015fjs})} 
                    \label{twins}
For a graph $G$ without isolated vertices, $B(G)=2$ if and only if $G$ has a pair of twins.
\end{proposition}

For a positive integer $n$, denote $[n]$ the set $\{1,2,\ldots,n\}$.
For integers $a\leq b$, denote $[a,b]$ the set $\{a,a+1,\ldots,b\}$.
The {\it symmetric difference} of two sets $A$ and $B$ is the set $A\Delta B=(A\cup B)\setminus (A\cap B)$.
For integers $n\ge k>a\geq 0$,
the {\it generalized Kneser graph} is the graph $K(n,k,a)$
whose vertices are the $k$-subsets of $[n]$
and two vertices $A$ and $B$ are adjacent if $|A\cap B|\leq a$,
and the {\it generalized Johnson graph} is the graph $J(n,k,a)$
whose vertices are the $k$-subsets of $[n]$
and $A$ and $B$ are adjacent if $|A\cap B|=a$.
The graphs $K(n,k,0)$ were known as Kneser graphs $K(n,k)$ and $J(n,k,k-1)$ as Johnson graphs $J(n,k)$.
Results for the zero blocking numbers of Kneser graphs and Johnson graphs were established by Afzali, Ghodrati and Maimani \cite{2024agm}.

\begin{theorem} {\bf (\cite{2024agm})}  
                \label{2024agm}
If $k\ge 2$ and $n\ge 2k+1$, then $B(K(n,k))=k+2$.
If $k\ge 1$ and $n\ge 2k$, then $B(J(n,k))=k+1$ except that $B(J(4,2))=2$.
\end{theorem}

Notice that $B(K(n,k))=1$ or $2$ when $k=1$ or $k\le n\le 2k$, since $K(n,1)=K_n$, $K(2k,k)=\frac{1}{2}{2k\choose k}K_2$ and $K(n,k)={n\choose k}K_1$ for $k\le n<2k$.
Also, $B(J(n,k))=1$ or $2$ when $k=1$ or $k=n$, since $J(n,1)=K_n$ and $J(k,k)=K_1$.
The following gives $J(n,k)$ when $k+1\le n<2k$.

\begin{proposition} \label{change k a}
If $n\ge 2k-a$, then $B(J(n,k,a))=B(J(n,n-k,n-2k+a))$.
Consequently, if $k+1\le n<2k$, then $B(J(n,k))=n-k+1$.
\end{proposition}
\begin{proof}
The first part follows from the fact that $J(n,k,a)$ is isomorphic to  $J(n,n-k,n-2k+a)$
under the mapping $A\mapsto [n]\setminus A$.

For the second part, since $n-k\ge 1$ and $n>2(n-k)$, by Theorem \ref{2024agm}, $B(J(n,k))=B(J(n,n-k))=n-k+1$.
\end{proof}

This paper extends the results by Afzali, Ghodrati and Maimani for the zero blocking numbers of Kneser graphs and Johnson graphs to generalized Kneser graphs and generalized Johnson graphs.
In particular, it is proved that
$B(K(n,k,a))=\lceil\frac{k-a}{a+1}\rceil+2$ (when $a\le k-2$) and $B(J(n,k,a))=\max\{a,k-a\}+2$ when $n$ is relatively larger than $k$ and $a$.

\section{Results for generalized Kneser graphs $K(n,k,a)$}

First, all $K(n,k,a)$ with $B(K(n,k,a))$ no more than two are characterized.

\begin{proposition} \label{small}
Suppose integers $n\geq k>a\geq 0$.
\begin{enumerate}
\itemsep -3pt
\item[$(1)$]
If $n\le 2k-a-1$, then $B(K(n,k,a))=1$.

\item[$(2)$]
If $a=k-1$, then $B(K(n,k,a))=\min\{n-k+1,2\}$.

\item[$(3)$]
If $n=2k$ and $a=0$, then $B(K(n,k,a))=2$.

\item[$(4)$]
If $n\ge 2k-a$ and $k-2\ge a$ but $(n,a)\ne (2k,0)$, then $B(K(n,k,a))\ge 3$.
\end{enumerate}
\end{proposition}
\begin{proof}
(1) This follows from $K(n,k,a)={n\choose k}K_1$.

(2) This follows from $K(n,k,k-1)= K_{n\choose k}$.

(3) This follows from $K(2k,k,0)= \frac{1}{2}{2k\choose k}K_2$.

(4) Since $n\ge 2k-a$, every vertex $A$ is adjacent to some vertex $B$ with $|A\cap B|=a$. Hence $K(n,k,a)$ has no isolated vertices.

Next to see that $K(n,k,a)$ has no twins.
Consider two distinct vertices $A$ and $B$.
Choose $x\in A\setminus B$ and $y\in B\setminus A$.
\begin{enumerate}
\itemsep -1pt
\item[(i)] Suppose $A\cap B=\emptyset$.
If $n>2k$, then there exists $z\in [n]\setminus(A\cup B)$.
In this case, $C=B\Delta\{y,z\}\notin\{A,B\}$ is adjacent to $A$ but not to $B$, since $|C\cap A|=0\leq a$ and $|C\cap B|=k-1>a$.
Otherwise, $n=2k$ and so $a\ge 1$, since $(n,a)\neq(2k,0)$.
Then $C=B\Delta\{x,y\}\notin\{A,B\}$ is adjacent to $A$ but not to $B$, since $|C\cap A|=1\leq a$ and $|C\cap B|=k-1>a$.

\item[(ii)] Suppose $0<|A\cap B|\leq a$.
Since $A\cap B\neq \emptyset$, there exists $w\in A\cap B$.
Then $C=B\Delta\{x,w\}\notin\{A,B\}$ is adjacent to $A$ but not to $B$, since $|C\cap A|=|(A\cap B)\Delta\{x,w\}|=|A\cap B|\leq a$ and $|C\cap B|=k-1>a$.

\item[(iii)] Suppose $a< |A\cap B|\leq k-1$.
Choose $C'\subsetneqq A\cap B$ with $|C'|=a$,
and $C''\subseteq [n]\setminus A$ with $B\setminus A\subsetneqq C''$ and $|C''|=k-a$
by the fact that $|B\setminus A|<k-a$.
Then $C=C'\cup C''\notin\{A,B\}$ is adjacent to $A$ but not to $B$, since $|C\cap A|=|C'|=a$ and $|C\cap B|=|C'|+|B\setminus A| \geq a+|\{y\}|=a+1$.
\end{enumerate}

In summary, $K(n,k,a)$ has no isolated vertices and has no twins.
These give that $B(K(n,k,a))\ge 3$ by Propositions \ref{isolated} and \ref{twins}.
\end{proof}

Next we consider $B(K(n,k,a))$ for $n\ge 2k-a$ and $k-2\ge a$ but $(n,a)\ne (2k,0)$.
We are only able to determine $B(K(n,k,a))$ for $n$ relatively larger than $k$ and $a$.

\begin{lemma} \label{Lemma 1}
If $m$ is a positive integer such that
$n\geq mk$ and $k+1\le (a+1)(m-1)$,
then $B(K(n,k,a))\le m$.
\end{lemma}
\begin{proof}
Let ${\cal W}=\{W_1,W_2,\ldots,W_m\}$ be a set of $m$ disjoint $k$-subsets of $[n]$.
Consider ${\cal W}$ as the set of white vertices.
If there exists a black vertex $S$
adjacent to exactly one vertex in ${\cal W}$, say $|S\cap W_1|\leq a$ and
$|S\cap W_i|\geq a+1$ for $2\leq i\leq m$.
Then $k=|S|\geq (a+1)(m-1)$, a contradiction.
Hence, ${\cal W}$ is a zero blocking set, which gives $B(K(n,k,a))\le m$.
\end{proof}

\begin{lemma} \label{Lemma 2}
If $n\ge 2k-a$ and $k-2\ge a$ but $(n,a)\ne (2k,0)$,
then $B(K(n,k,a))\geq \lceil\frac{k-a}{a+1}\rceil+2$.
\end{lemma}
\begin{proof}
Suppose $n=2k-a$.
For every two vertices $A$ and $B$, we have $|A\cap B|=|A|+|B|-|A\cup B|\ge k+k-n=a$ and so $K(2k-a,k,a)=J(2k-a,k,a)$.
By Proposition \ref{change k a} and Theorem \ref{2024agm},
$B(K(2k-a,k,a))=B(J(2k-a,k,a))=B(J(2k-a,k-a,0))=B(K(2k-a,k-a))=k-a+2
\geq \lceil\frac{k-a}{a+1}\rceil+2$.

Now $n>2k-a$.
Suppose to the contrary that $K(n,k,a)$ has
a minimum zero blocking set $\{W_1,W_2,\ldots,W_m\}$ with
$m\le \lceil\frac{k-a}{a+1}\rceil+1$, that is, $k\ge (a+1)(m-1)$.
By Proposition \ref{small} (4), $m\ge 3$.
For $i\ge 2$, let $d_i=|W_i\setminus W_1|$ and $d_0=\min_{i\ge 2} d_i>0$,
and choose $Y_i\subseteq W_i\setminus W_1$ with $|Y_i|=\min\{a+1,d_i\}$.
Also, for $d_0\le d\le k$, let $X_d=\bigcap_{d_i\le d} (W_i\cap W_1)$.
Then $X_k\subseteq X_{k-1}\subseteq\cdots\subseteq X_{d_0}$
and
\begin{eqnarray*}
  |X_d| &=&    |W_1\setminus \bigcup\nolimits_{d_1\le d} (W_1\setminus W_i)| \\
        &\geq& |W_1|-\sum\nolimits_{d_i\le d} |W_1\setminus W_i|
         =     |W_{1}|-\sum\nolimits_{d_i\le d} |W_i\setminus W_1| \\
        &\geq&  k-d(m-1)
         \geq  (a+1)(m-1)-d(m-1) \\
        &=&    (a+1-d)(m-1)
        \geq   a+1-d.
\end{eqnarray*}
Hence, for each $d_0\leq d\leq k$, we may
choose $X'_d\subseteq X_d$ with $|X'_d|=\max\{0,a+1-d\}$
such that $X_k'\subseteq X_{k-1}'\subseteq\cdots\subseteq X_{d_0}'$.
Let $Z=(\bigcup_{i=2}^{m} Y_{i})\cup X'_{d_{0}}$.
Then
$$
   |Z\cap W_1|=|X'_{d_0}|=\max\{0,a+1-d_0\}\le a.
$$
For $i\ge 2$, $|Y_i \cup X'_{d_i}|=
   \min\{a+1,d_i\}+\max\{0,a+1-d_i\}=a+1$ and so
$$
   |Z\cap W_i|\geq |Y_i \cup X'_{d_i}|=a+1.
$$
Also, $|Z|=|\bigcup_{i=2}^m (Y_i \cup X'_{d_i})|
          \le \sum_{i=2}^m |Y_i \cup X'_{d_i}|
          =   (a+1)(m-1)\le k$.

Since $n>2k-a$,
we can extend $Z$ to a $k$-subset $Z' \subseteq [n]$ with $|Z'\cap W_1| \le a$.
Also, $|Z'\cap W_i|\ge |Z\cap W_i|\ge a+1$ for $i\ge 2$.
Since $Z'$ is adjacent to exactly one white vertex $W_1$,
$Z'$ must be white, say $Z'=W_2$.
Similarly, $W_2$ has a white neighbor, which must be $W_1$,
not adjacent to other white vertices.

Notice that $W_2=Z'\supseteq Z=(\bigcup_{i=2}^m Y_i)\cup X'_{d_0}$.
If we extend $\bigcup_{i=3}^m (Y_i\cup X'_{d_i})$ to a $k$-subset $Z''$
with $|Z''\cap W_1|\le a$, then $|Z''\cap W_2|\ge |Y_3\cup X'_{d_3}|=a+1$ and
$|Z''\cap W_i|\ge |Y_i\cup X'_{d_i}|=a+1$ for $i\ge 3$.
Hence $Z''$ is adjacent to exactly one white vertex $W_1$ and so $Z''=W_2$.
But $|\bigcup_{i=3}^m (Y_i\cup X'_{d_i})|
\le \sum_{i=3}^m |Y_i\cup X'_{d_i}|
\le (a+1)(m-2)\le k-(a+1)<k$.
Since $n>2k-a$, there is more than one way to extend $\bigcup_{i=3}^m (Y_i\cup X'_{d_i})$ to $Z''$, a contradiction.

Hence, $B(K(n,k,a))\geq \lceil\frac{k-a}{a+1}\rceil+2$ as desired.
\end{proof}

\begin{theorem} \label{exact value}
If $n\geq k(\lceil\frac{k-a}{a+1}\rceil+2)$ and $k-2\ge a$,
then $B(K(n,k,a))=\lceil\frac{k-a}{a+1}\rceil+2$.
\end{theorem}
\begin{proof}
The theorem follows from  Lemma \ref{Lemma 1} by taking $m=\lceil\frac{k-a}{a+1}\rceil+2$ and Lemma \ref{Lemma 2} by
noting that $n\ge k(\lceil\frac{k-a}{a+1}\rceil+2)\ge 3k$.
\end{proof}

Now the exact value of $B(K(n,k,a))$ is determined for sufficient large $n$, but
not for small $n$.
Since $B(K(n,k,0))=k+2$ for $n\ge 2k+1$ and $k\ge 2$,
the bound of $n$ in Theorem \ref{exact value} is not sharp.
The following is a weaker upper bound of $B(K(n,k,a))$,
but true for arbitrary $n$.

\begin{theorem} \label{weak bound}
$B(K(n,k,a))\leq k-a+2$.
\end{theorem}
\begin{proof}
The theorem is true for $a=0$ or $n=k$.
Now assume that $a\ge 1$ and $n\ge k+1$.
Consider the set ${\cal W}=\{W_1,W_2,\ldots,W_{k-a+2}\}$,
where $W_i=[k+1]\setminus\{i\}$ for $1\leq i\leq k-a+2$.
Let ${\cal W}$ be the set of all white vertices.
To show that ${\cal W}$ is a zero blocking set.
Suppose $S$ is a black vertex adjacent to some white vertex $W_j$.

If $|S\cap W_j|\le a-1$, then for $j'\neq j$,
$|S\cap W_{j'}|\leq |(S\cap W_{j})\cup\{j\}|\leq (a-1)+1=a$.

If $|S\cap W_j|=a$, then there exists $j'\in S\cap W_j$ such that $j'\le k-a+2$, since $|[k-a+3, k+1]|=a-1$.
In this case, $j'\ne j$ and $S$ is adjacent to $W_{j'}$, since
$|S\cap W_{j'}|\leq|((S\cap W_{j})\cup\{j\})\setminus\{j'\}|\le a+1-1=a$.

Hence, ${\cal W}$ is a zero blocking set as desired.
The theorem then follows.
\end{proof}

Consequently, $B(K(n,k,a))$ is monotonically non-increasing in $n$ for $n\ge 3k-2a+1$.

\begin{theorem}
If $n\ge 3k-2a+1$, then $B(K(n+1,k,a))\leq B(K(n,k,a))$.
\end{theorem}
\begin{proof}
If $\{W_1,W_2,\ldots,W_m\}$ is a minimum zero blocking set of $K(n,k,a)$
but not a zero blocking set of $K(n+1,k,a)$,
then $K(n+1,k,a)$ has a black vertex $S$ containing $n+1$ that is adjacent to exactly one white vertex, say $W_1$.
By Theorem \ref{weak bound}, $m\le k-a+2$.
For every $S'=S\Delta\{x,n+1\}$ with $x\in [n]\setminus S$ and $|S'\cap W_{1}|\le a$,
we have $|S'\cap W_i|\ge |S\cap W_i|>a$ for $i\ge 2$.
Since $n\geq k+(k-1-a)+(k-a+2)\geq |W_{1}|+(|S\setminus\{n+1\} |-a)+m$,
there exists at least $m$ choices of $S'$,
so that one of them is a black vertex in $K(n,k,a)$ that forces $W_1$ to be black, a contradiction.
Hence, $\{W_1,W_2,\ldots,W_m\}$ is also a zero blocking set of $K(n+1,k,a)$ and
so $B(K(n+1,k,a))\leq B(K(n,k,a))$.
\end{proof}

\section{Results for generalized Johnson graphs $J(n,k,a)$}

First, all $J(n,k,a)$ with $B(J(n,k,a))$ no more than two are characterized.

\begin{proposition} \label{small2}
Suppose integers $n\geq k>a\geq 0$.
\begin{enumerate}
\itemsep -3pt
\item[$(1)$]
If $n\le 2k-a-1$, then $B(J(n,k,a))=1$.

\item[$(2)$]
If $n\geq 2$ and $k=1$, or $n=k+1=a+2$, then $B(J(n,k,a))=2$.

\item[$(3)$]
If $n=2k$ and $a=0$, then $B(J(n,k,a))=2$.

\item[$(4)$]
If $n=2k=4a>0$, then $B(J(n,k,a))=2$.

\item[$(5)$]
If $n\ge 2k-a$ and $k\ge 2$ but $(n,a)\ne (2k,0)$ and it is not the case that $n=2k=4a>0$
or $n=k+1=a+2$, then $B(J(n,k,a))\ge 3$.
\end{enumerate}
\end{proposition}
\begin{proof}
(1) This follows from $J(n,k,a)={n\choose k} K_1$.

(2) This follows from $J(n,1,0)=J(n,n-1,n-2)=K_n$.
Notice that the case of $n=k=1$ is included in (1).

(3) This follows from $J(2k,k,0)= \frac{1}{2}{2k\choose k}K_2$.

(4) Choose two vertices $A$ and $B$ with $A\cap B=\emptyset$.
If vertex $C$ is adjacent to $A$, then $|C\cap A|=a$ and so
$|C\cap B|=|C|-|C\cap A|=k-a=a$, implying that $C$ is adjacent to $B$.
Similarly, if $C$ is adjacent to $B$, then $C$ is adjacent to $A$.
Hence $A$ and $B$ are twins.
This gives $B(J(n,k,a))=2$.

(5) Since $n\ge 2k-a$, every vertex $A$ is adjacent to some vertex $B$ with $|A\cap B|=a$.
Hence $J(n,k,a)$ has no isolated vertices.

Next to see that $J(n,k,a)$ has no twins.
Suppose to the contrary that there are twins $A$ and $B$.
Choose $x\in A\setminus B$ and $y\in B\setminus A$.
\begin{enumerate}
\itemsep -1pt
\item[(i)] Suppose $A\cap B=\emptyset$.
Choose $A'\subseteq A$ with $|A'|=a$
and $B'\subseteq B$ with $y\in B'$ and $|B'|=k-a$.
Then $C=A'\cup B'$ is adjacent to $A$, and so $C=B$ or $C$ is adjacent to $B$.
If $C=B$, then $a=0$ and so $n>2k$, since $(n,a)\neq (2k,0)$.
If $C$ is adjacent to $B$, then $k-a=a>0$ and so $n>2k$, since it is not the case $n=2k=4a>0$.
Hence there exists $z\in [n]\setminus(A\cup B)$.
Then $C'=C\Delta \{y,z\}\notin \{A,B\}$ is adjacent to $A$ but not to $B$, since
$|C'\cap A|=|A'|=a$, $|C'\cap B|=|B\setminus\{y\}|=k-1\geq 1> 0=a$ when $C=B$,
and $|C'\cap B|=|B'\setminus \{y\}|=k-a-1=a-1$ when $C$ is adjacent to $B$.

\item[(ii)] Suppose $0<|A\cap B|\leq a$.
Since $A\cap B \neq\emptyset$, there exists $w\in A\cap B$.
Choose $A'$ with $((A\cap B)\cup\{x\})\setminus\{w\}\subseteq A'\subseteq A\setminus\{w\}$
and $|A'|=a$, and
$B'\subseteq B\setminus A$ with $|B'|=k-a$ by the fact that $|B\setminus A|\geq k-a$.
Vertex $C=A'\cup B'\notin \{A,B\}$ is adjacent to $A$, since $|C\cap A|=|A'|=a$.
Hence $C$ is adjacent to $B$, that is $|C\cap B|=a$.
Then $C'=C\Delta\{x,w\}\neq A$ is adjacent to $A$ but not to $B$, since $|C'\cap A|=|A'\Delta\{x,w\}|=|A'|=a$ and
$|C'\cap B|=|(C\cap B)\cup \{w\}|=a+1$.
Hence $C'=B$, that is $C=B\Delta\{x,w\}$,
which implies $a=|A\cap B|=|C\cap B|=k-1$
and $A\setminus \{x\}=A\cap B=B\setminus\{y\}$
with $|A\cup B|=k+1$.
Since it is not the case $n=k+1=a+2$,
there exists $z\in [n]\setminus(A\cup B)$.
Hence $C''=C\Delta \{y,z\}\neq \{A,B\}$ is adjacent to $A$ but not to $B$, since $|C''\cap A|=|A'|=a$ and $|C''\cap B|=|B\setminus\{w,y\}|=k-2=a-1$.

\item[(iii)] Suppose $a< |A\cap B|\leq k-1$.
Choose $C'\subsetneqq A\cap B$ with $|C'|=a$,
and $C''\subseteq [n]\setminus A$ with $B\setminus A\subsetneqq C''$ and $|C''|=k-a$ by the fact that
since $|B\setminus A|<k-a$.
Then $C=C'\cup C''\notin\{A,B\}$ is adjacent to $A$ but not to $B$, since $|C\cap A|=|C'|=a$ and $|C\cap B|=|C'|+|B\setminus A| \geq a+|\{y\}|=a+1$.
\end{enumerate}

In summary, $J(n,k,a)$ has no isolated vertices and has no twins.
These give that $B(J(n,k,a))\ge 3$ by Propositions \ref{isolated} and \ref{twins}.
\end{proof}

Notice that $J(n,k,0)=K(n,k)$ and $J(n,k,k-1)=J(n,k)$.
So we may consider $J(n,k,a)$ only for $k-2\ge a\ge 1$.

\begin{theorem} \label{upper bound}
If $a\ge 1$, then $B(J(n,k,a))\le \max\{a,k-a\}+2$.
\end{theorem}
\begin{proof}
The theorem is true for $n=k$, since $B(J(k,k,a))=1$ by Proposition \ref{small2} (1).
Now assume that $n\ge k+1$.

Let $m=\max\{a,k-a\}+2$.
Consider the set ${\cal W}=\{W_1,W_2,\ldots,W_m\}$ of $k$-subsets of $[n]$,
where $W_i=[k+1]\setminus \{i\}$ for $1\leq i\leq m$.
Let ${\cal W}$ be the set of all white vertices.
We verify that ${\cal W}$ is a zero blocking set by showing that
any black vertex $S$ adjacent to some $W_i$ is adjacent to another $W_{i'}$ with $i'\ne i$ as follows.
Since $|S\cap W_i|=a$ and $|[m+1,k+1]|<a<|[m]\setminus\{i\}|$,
there exists $i_1,i_2\in [m]\setminus\{i\}$ such that
$i_1\in S\cap W_i$ and $i_2\in W_i\setminus S$.
If $i\in S$, then $|S\cap W_{i_1}|=|S\cap (W_i\Delta \{i,i_1\})|=a+1-1=a$.
If $i\notin S$, then $|S\cap W_{i_2}|=|S\cap (W_i\Delta \{i,i_2\})|=|S\cap W_i| =a$.
Hence, $\cal W$ is a zero blocking set of $J(n,k,a)$ and so
$B(J(n,k,a)) \le m$ as desired.
\end{proof}

By Proposition \ref{change k a},
$B(J(n,k,a))=B(J(n,n-k,n-2k+a))\leq \max\{k-a,n-2k+a\}+2$ for $n\ge 2k-a$.
This upper bound is smaller than $\max\{a,k-a\}+2$ for $2k-a\le n<2k$ and $a>k/2$.
Although the exact value of $B(J(n,k,a))$ is not known in general, the following shall prove that the upper bound $\max\{a,k-a\}+2$ is in fact the value of $B(J(n,k,a))$ when $n$ is relatively larger than $k$ and $a$.

First, two useful lemmas.

\begin{lemma} \label{bound}
In $J(n,k,a)$,
if vertices $A$ and $B$ have a common neighbor $C$,
then $|A\setminus B|\le 2(k-a)$.
\end{lemma}
\begin{proof}
Since $A\setminus B \subseteq (A\setminus C) \cup (C\setminus B)$,
it follows that $|A\setminus B| \le |A\setminus C| + |C\setminus B|=2(k-a)$.
\end{proof}

\begin{lemma} \label{S1 S2}
Suppose $S_1,S_2,\ldots,S_p$ are $p\ge 2$ pairwise disjoint sets and $a_1,a_2,\ldots,a_p$ are positive integers.
If $|S_1|\ge 3$ and $|S_2|\ge 2$,
and $2\le t\le (\sum_{i=1}^p|S_i|)-p-1$,
then there exists $S\subseteq \bigcup_{i=1}^{p} S_{i}$
such that $|S|=t$ and $|S\cap S_{i}|\neq a_{i}$ for $1\leq i\leq p$.
\end{lemma}
\begin{proof}
Choose $x_i\in S_i$ for $S_i\ne\emptyset$ and $y_1\in S_1\setminus \{x_1\}$.
Since $2\le t\le (\sum_{i=1}^p|S_i|)-p-1$,
there exists $S_0\subseteq (\bigcup_{i=1}^p (S_i\setminus \{x_i\})) \setminus \{y_1\}$
such that $|S_0|=t$ and $S_0\cap S_i$ contains some $z_i$ for $i=1,2$.
Let $I=\{i : 1\le i\le p, |S_0\cap S_i|=a_i\}$.
Since $a_i>0$, there exists $w_i\in S_0\cap S_i$ for $i\in I$.
Suppose $I=\{i_1,i_2,\ldots,i_r\}$.

If $r=2s$ is even, then $S = (S_0 \cup \{x_{i_1}, x_{i_2}, \ldots, x_{i_s}\})\setminus \{w_{i_{s+1}},w_{i_{s+2}},\ldots,w_{i_{2s}}\}$ is as desired.
Now suppose $r=2s+1$ is odd with $s\ge 0$.

If $1\notin I$ or $2\notin I$, say $i_0\in\{1,2\}$ but $i_0\notin I$.
If $|S_0\cap S_{i_0}|+1\ne a_{i_0}$, then $S = (S_0 \cup \{x_{i_0},x_{i_1}, x_{i_2}, \ldots, x_{i_s}\})\setminus \{w_{i_{s+1}},w_{i_{s+2}},\ldots,w_{i_{2s+1}}\}$ is as desired.
Otherwise, $|S_0 \cap S_{i_0}|-1\ne a_{i_0}$ and so $S = (S_0 \cup \{x_{i_1}, x_{i_2}, \ldots, x_{i_{s+1}}\})$ $\setminus \{w_{i_{s+2}},w_{i_{s+3}},\ldots,w_{i_{2s+1}},z_{i_0}\}$ is as desired.

Now suppose $1\in I$ and $2\in I$, say $1=i_1$, then $r=2s+1\ge 3$ and so
$S = (S_0 \cup \{y_1, x_{i_1}, x_{i_2},\ldots, x_{i_s}\})\setminus \{w_{i_{s+1}},w_{i_{s+2}},\ldots,w_{i_{2s+1}}\}$ is as desired.
\end{proof}

Now consider the case of $a\ge k/2$.

\begin{theorem}
If $n\ge (a+\frac52)(k-a+2)-6$ and $k-2\ge a\ge k/2$,
then $B(J(n,k,a))= a+2$.
\end{theorem}
\begin{proof}
By Lemma \ref{upper bound}, it is only necessary to prove that $B(J(n,k,a))\ge a+2$.
For this purpose, let $d=k-a$ and so $d\ge 2$.

Suppose to the contrary that $J(n,k,a)$ has a minimum zero blocking set $\{W_1,W_2,\ldots,W_m\}$ with $m\le a+1$.
Choose $x_i\in W_1\setminus W_i$ for $2\le i\le m$
and extend $\{x_2,x_3,\ldots,x_m\}$ to a set $X$
with $X\subseteq W_1$ and $|X|=a$.
Let $I=\{i\ge 2: |W_i\setminus W_1|\le 2d\}$
and $U=[n]\setminus \bigcup_{i\in \{1\}\cup I} W_i$.

If $|U|\ge d$, then choose $Y\subseteq U$ with $|Y|=d$.
Then, $|(X\cup Y)\cap W_1|=|X|=a$ and $x_i\in (X\cup Y)\setminus W_i$ for $2\le i\le m$.
Hence, $X\cup Y$ is a black vertex adjacent to $W_1$.
So, $X\cup Y$ must be adjacent to another white vertex $W_j$.
By Lemma \ref{bound}, $|W_j\setminus W_1|\le 2d$ and so $j\in I$.
By the definition of $U$ and $Y$, $Y\cap W_j=\emptyset$,
Hence $a=|(X\cup Y)\cap W_j|=|X\cap W_j|\le |X\setminus \{x_j\}|<a$, a contradiction.
In conclusion, $|U|\le d-1$.

For $i\in I$, let $W_i'=W_i\setminus W_1$ and $D_i=W_i'\setminus \bigcup_{j\ne i, j\in I} W_j'$.
Notice that $U=[n]\setminus (W_1\cup \bigcup_{i\in I} W_i')$.
It is clear that
$
   \sum_{i\in I}|W_i'| + \sum_{i\in I}|D_i| \ge 2|\bigcup_{i\in I} W_i'|.
$
Now, $|\bigcup_{i\in I} W_i'|=n-|W_1|-|U|\ge n-k-d+1$, $|I|\le m-1\le a$ and $|W_i'|\le 2d$ for $i\in I$.
In summary,
$\sum_{i\in I}|D_i| \ge 2|\bigcup_{i\in I} W_i'| -  \sum_{i\in I}|W_i'|
\ge 2(n-k-d+1)-2da \ge a+k$ by the assumption that $n\ge (a+\frac52)(k-a+2)-6$.

There exist $i,j\in I$ with $i\neq j$ such that $|D_i|\ge 3$ and $|D_j|\ge 2$,
since $\sum_{i\in I} |D_i|\ge a+k\geq |I|+\max\{|I|+1,|D_i|: i\in I\}$,
which also implies $|I|\geq 2$.
Also, $d\le a+k-a-1\le (\sum_{i\in I} |D_i|)-|I|-1$.
Now apply Lemma \ref{S1 S2}
by taking $\{S_1,S_2,\ldots,S_p\}=\{D_i: i\in I\}$,
$a_i=a-|W_i\cap X|$ for $i \in I$ and $t=d$,
to get $S\subseteq \bigcup_{i\in I} D_i$ with $|S|=d$
and $|S\cap D_i|\ne a_i$ for $i\in I$.
Then $X\cup S$ is a black vertex adjacent to $W_1$.
For $i\in I$, $|(X\cup S)\cap W_i|=|X\cap W_i|+|S\cap D_i|\ne a$.
For other $i \in [2,m]\setminus I$, $X\cup S$ is not adjacent to $W_{i}$ by Lemma \ref{bound}.
Hence the black vertex $X\cup S$ is adjacent to exactly one white vertex $W_1$,
a contradiction.
\end{proof}

Next consider the case of $a <k/2$.

\begin{theorem}
If $n\ge k(k-a+2)-a$ and $k/2>a\ge 1$, then $B(J(n,k,a))=k-a+2$.
\end{theorem}
\begin{proof}
By Lemma \ref{upper bound}, it is only necessary to prove that $B(J(n,k,a))\ge k-a+2$.
Suppose to the contrary that $J(n,k,a)$ has a minimum zero blocking set $\{W_1,W_2,\ldots,W_m\}$ with $m\le k-a+1$.
Choose $x_i\in W_1\setminus W_i$ for $2\le i\le m$.
Let $W_1'=\{x_2,x_3,\ldots,x_m\}$ and $W_1''=W_1\setminus W_1'$.
Assume that $W_1'$ is chosen so that $r=|W_1'|$ is minimum.
Notice that $r\le m-1\le k-a$ and so $|W_1''|\ge a$.

Choose $X_0\subseteq W_1''$ with $|X_0|=a$.
Let $X=X_0$ initially and update it at most $m-1$ times as follows.
If there exists $2\le i\le m$ such that $|W_i\cap X|=a$, then
$|W_i\cap (W_1\cup X)| \le |W_1\setminus (X\cup \{x_i\})|+|W_i\cap X|=(k-a-1)+a<|W_i|$ and so there is some $y_i\in W_i\setminus (W_1\cup X)$.
For such $i$, add $y_i$ into $X$.
Now $|W_i\cap X|>a$.
Repeating this process at most $m-1$ times results a set $X$ with $a\le |X|\le a+m-1\le k$.
Also $|W_1\cap X|=a \ne |W_i\cap X|$ for $2\le i\le m$.
Since $n\ge k(k-a+2)-a\ge km+k-a$, we can add $k-|X|$ elements in $[n]\setminus\bigcup_{i=1}^m W_i$ into $X$.
Now $|X|=k$ with $X\cap W_1=X_0\subseteq W_1''$ and $X$ is adjacent to exactly one white vertex $W_{1}$,
so $X$ must be white, say $X=W_2$.
Similarly, $W_2$ has a white neighbor, which must be $W_1$,
not adjacent to other white vertices.

Let $I=\{i: 2\le i\le m, W_i\cap W_1'\ne\emptyset\}$.
If $I=\emptyset$, then $r=1<k-a$, else $r\le |I|$ since $\{x_i: i\in I\}$ has the same property as $W_1'$ and $W_1'$ is chosen so that $r=|W_1'|$ is minimum.
Since $2\notin I$, we have $r\le |I|\le m-2<k-a$.
By the arguments in the previous paragraph, each subset $X_0\subseteq W_1''$ with $|X_0|=a$ corresponds to a white vertex $W_j$ with $2\le j\le m$ such that $W_j\cap W_1=X_0\subseteq W_1''$, adjacent to exactly one white vertex $W_1$, which must be $W_2$.
But there is more than one choice for $X_0$, since $r<k-a$, a contradiction.

The theorem then follows.
\end{proof}

\bigskip
\noindent
{\bf Acknowledgement.}
This work was supported in part by the National Science and Technology Council in
Taiwan under grant NSTC 113-2115-M-A49-013.

\end{document}